\documentclass[preprint,12pt]{article}
\usepackage{epsfig}
\usepackage{amsmath}
\usepackage{amssymb}
\usepackage{amsthm}
\usepackage{multicol}
\usepackage{authblk}
\theoremstyle{plain}
\newtheorem{theorem}{Theorem}[section]
\newtheorem{corollary}[theorem]{Corollary}
\newtheorem{lemma}[theorem]{Lemma}

\theoremstyle{remark}

\theoremstyle{definition}
\newtheorem{definition}[theorem]{Definition}
\newcommand{\R}{\mathbb{R}}

\title{Well-posed problem for a combustion model in a multilayer porous medium}

\author[1]{Marcos R. Batista\thanks{Instituto Federal de Educa\c c\~ao, Campus Goi\^ania, GO, $74055-110,$ Brazil}, Alysson Cunha\thanks{Instituto de Matem\'atica e Estat\'istica - Universidade Federal de Goi\'as, Campus Samambaia, Goi\^ania , GO, $74690-900,$ Brazil}, Jesus C. Da Mota\thanks{Instituto de Matem\'atica e Estat\'istica - Universidade Federal de Goi\'as, Campus Samambaia, Goi\^ania , GO, $74690-900,$ Brazil}, Ronaldo A. Santos\thanks{Instituto de Matem\'atica e Estat\'istica - Universidade Federal de Goi\'as, Campus Samambaia, Goi\^ania , GO, $74690-900,$ Brazil}}

\begin{document}

\maketitle

\begin{abstract}
Combustion occurring in porous media has various practical applications, such as in in-situ combustion processes in oil reservoirs, the combustion of biogas in sanitary landfills, and many others. A porous medium where combustion takes place can consist of layers with different physical properties. This study demonstrates that the initial value problem for a combustion model in a multi-layer porous medium has a unique solution, which is continuous with respect to the initial data and parameters in $\mathtt{L}^2(\R)^n$. In summary, it establishes that the initial value problem is well-posed in $\mathtt{L}^2(\R)^n$. The model is governed by a one-dimensional reaction-diffusion-convection system, where the unknowns are the temperatures in the layers. Previous studies have addressed the same problem in $\mathtt{H}^2(\R)^n$. However, in this study, we solve the problem in a less restrictive space, namely $\mathtt{L}^2(\R)^n$. The proof employs a novel approach to combustion problems in porous media, utilizing an evolution operator defined from the theory of semigroups in Hilbert space and Kato's theory for a well-posed associated initial value problem.
\end{abstract}

	
	

% {\bf MSC:} 35K51, 35K57, 76S05, 80A25.

\section{Introduction}\label{introduction}
Combustion within porous media has extensive practical
applications, and there exists a vast array of both theoretical and
experimental literature in this field. We reference here some
recent works in this area, which are by no means exhaustive \cite{Banerjee, Fatehi, Gharehghani, Mujeebu, Trimis}.

One technique utilized for extracting oil from a petroleum
reservoir (a porous medium) is referred to as ``in situ combustion"
\cite{Ado, Chapiro, Coats, Sarathi}. This method involves the propagation of a burning flame within the reservoir, which reduces the viscosity of the oil and facilitates its flow towards the producing wells.

Porous media where crude oil is present may consist of multiple layers \cite{Gao, Lefkovits}, 
each characterized by variations in properties such as porosity, density, and thermal conductivity, among others.

In \cite{DaMota-Schecter}, a model was formulated to investigate the propagation of a combustion front through a porous medium containing two parallel layers with different properties. The reaction involves oxygen and a premixed solid fuel in each layer. The model consists of a coupled nonlinear reaction-diffusion-convection system along with a set of ordinary differential equations. These equations are derived from balance equations and Darcy’s law. Assuming incompressibility, the variables are reduced to the temperatures and unburned fuel concentrations in the layers. This study presents the discovery of a family of traveling wave solutions that connect a burned state behind the combustion front to an unburned state ahead of it. 

In reference \cite{DaMota-Marcelo}, the Cauchy problem for the two-layer model, specifically in cases where the fuel concentrations are known functions, has been resolved. The iterative method of upper and lower solutions was employed to establish the existence and uniqueness of a classical solution.

The existence of a classical solution to the Cauchy problem for the complete two-layer model, where temperatures and fuel concentrations are unknown, was established in \cite{DaMota-Marcelo-Ronaldo}. The employed techniques involved iterations of the fundamental solutions to an associated linear problem.

In \cite{Batista1}, a generalization of the two-layer model for any number of layers ($n$-layers) was presented. This model takes into account the heat loss to the external environment, which was neglected in \cite{DaMota-Schecter}. A summary of this $n$-layer model is provided in Appendix \ref{model}. The existence and uniqueness of a classical solution to an initial and boundary value problem associated with the model was proven in\cite{DaMota-Schecter} for the case where the fuel concentration in each layer is a known function. The proof also involves monotone iterations of upper and lower solutions conveniently constructed for the problem. 

With regards to the complete system, where the temperatures and fuel concentrations in the $n$ layers are unknown, the existence of a classical solution to the initial and boundary problem was
proven in \cite{Batista2}. The local solution was obtained by defining an operator in a set of H\"older continuous functions, and Schauder's fixed-point theorem was used to find a fixed point as the desired solution. Using Zorn's lemma, the local solution was extended to the global-in-time solution. However, uniqueness and continuous dependence were not proven in \cite{Batista2}.

In the present work, we consider the $n$-layer model in a more realistic case, where physical properties such as porosity, initial fuel concentration, and thermal capacities are all functions of the spatial variable $x$, instead of being constants as in previous studies. However, it should be noted that in our study, we assume that the fuel concentrations are known functions. If this assumption does not hold and the fuel concentrations are unknown, the theory we utilize, mentioned below, cannot be directly applied. This is because the partial differential equation (PDE) system described by Eqs.\eqref{be14}-\eqref{bmf4} is not of the reaction-diffusion type in that case. 

In summary, we prove that the following initial value problem is well-posed in $\mathtt{L}^2(\R)^n$ using the abstract semigroups theory of operators in a Hilbert space and Kato's theory for a well-posed initial value problem. Recently, in \cite{Alarcon}, the same problem was proven in $\mathtt{H}^2(\R)^n$, a more restrictive space than $\mathtt{L}^2(\R)^n$.

Thus, this paper is concerned with the following Cauchy problem 

\begin{align}\label{combustion-problem}
\left\{
\begin{array}{l}
u_t + L(t)u = f(x, t, u), \,\,\, x \in \R, \,\,\, t >0,  \\
u(x, 0) = \phi(x), 
\end{array}
\right.
\end{align}
where $u = (u_1, \ldots, u_n)$ is the vector of the unknown temperatures, $u_i = u_i(x, t)$ is the temperature of the layer $i$, $\phi = (\phi_1, \ldots, \phi_n)$ is the vector of the initial temperatures and $\phi_i = \phi_i(x)$ is the initial temperature of the layer $i$, for $i = 1, \ldots, n$ . The partial differential operator $L(t)$ is defined by
\begin{align}\label{operatorL} 
L(t)u = \big(L_1(t)u_1, \ldots, L_n(t)u_n \big),
\end{align}
where
\begin{equation}\label{operatorLi}
L_i(t)u_i := -\alpha_i(x, t)\,\partial_{xx}u_i + \beta_i(x, t)\,\partial_{x}u_i,\,\,\,
i = 1, \ldots, n.
\end{equation}
The coefficients $\alpha_i$, $\beta_i$, and the source function $f =(f_1, \ldots, f_n)$, are defined in Appendix \ref {model}. The deduction of the complete model can be seen in \cite{Batista1}. Specifically, we have
\begin{equation}\label{coeficients}
\alpha_i(x, t) = \frac{\lambda_i}{a_i + b_i y_i(x, t)}, \,\,\, \mbox{and} \,\,\, \beta_i(x, t) = 
\frac{c_i}{a_i + b_i y_i(x, t)}, 
\end{equation}
 
\begin{align} \label{reaction-function-1} 
	&f_1(x, t, u) = \frac{-(c_1)_x\,u_1 }{a_1 + b_1 y_1} + 
	\frac{(K_1 b_1 u_1+d_1)y_1\,g(u_1)}{a_1 + b_1 y_1} + \frac{q_1(u_2 - u_1)}{a_1 + b_1 y_1}\
	\nonumber
	\\   
	& \hspace{0.8in} - \frac{\hat q_1(u_1 - u_e)}{a_1+b_1 y_1}\,,
	\nonumber \\
	&f_i(x, t, u) = \frac{-(c_i)_x\,u_i }{a_i + b_i y_i} + 
	\frac{(K_i b_i u_i+d_i)y_i\,g(u_i)}{a_i + b_i y_i} - \frac{q_{i-1}(u_i - u_{i-1})}{a_i + b_i y_i}
	\\
	& \hspace{0.8in} + \frac{q_i(u_{i+1} - u_i)}{a_i + b_i y_i}, \quad i = 2, \ldots, n-1\,,
	\nonumber \\
	&f_n(x, t, u) = \frac{-(c_n)_x\,u_n }{a_n + b_n y_n} + 
	\frac{(K_n b_n u_n+d_n)y_n\,g(u_n)}{a_n + b_n y_n} - \frac{q_{n-1}(u_n - u_{n-1})}{a_n + b_n y_n}
	\nonumber
	\\
	& \hspace{0.8in} - \frac{\hat q_2(u_n - u_e)}{a_n + b_n y_n}\,. 
	\nonumber
\end{align}
Here, the functions $a_i$, $b_i$, $c_i$, $d_i$, $\lambda_i$, $q_i$, $K_i$, $\hat q_1$, and 
$\hat q_2$, depend on the physical properties of the porous medium layers, as described by Eqs. \eqref{parametros1}, \eqref{parametros2},
and \eqref{parametros3}, and they are all nonnegative functions of the
spatial variable $x$. In this context, $a_i$ and $\lambda_i$ are considered positive. The function $g$ follows the Arrhenius law, given by Eq. \eqref{darcy}. The temperature of the external environment is denoted by
$u_e$, which is assumed to be constant. Finally, the fuel concentration $y_i$ is a known nonnegative and bounded\footnote{We can reasonably assume that the fuel concentrations $y_i$ take their values in the interval $[0, 1]$, as this is physically expected, in view of the concentration definition.} 
function of $(x, t)$.

Throughout this study, the index $i$ refers to the layer $i$ of the porous medium, and unless stated otherwise $i = 1, \ldots, n$.
When there is no ambiguity, we will omit the variable $x$ from $u(x, t)$, $\phi(x)$, $y(x,t)$, $f(x, t, u)$, $\alpha_i(x, t)$, and $\beta_i(x, t)$ by simply using $u(t)$, $\phi$, $y(t)$, $f(t, u)$, $\alpha_i(t)$, and $\beta_i(t)$.

Direct application of the semigroup theory to Problem
\eqref{combustion-problem} is not possible because the operator
$L(t)$ is not closed in a Hilbert space. Therefore, let us consider
the following associated problem:
\begin{equation}\label{combustion-problem3}\left\{
\begin{array}{l}
u_t + G(t)u = f(t, u), \,\,\, t > 0, \\
u(0) = \phi.
\end{array}
\right.
\end{equation}
Here, $G(t)$ is the closure of $L(t)$ in $\mathtt{L}^2(\R)^n$, it is defined by
$$G(t)u=\big(G_1(t)u_1, \ldots, G_n(t)u_n \big),$$ where 
$G_i(t)$ is the closure of $L_i(t)$ in $\mathtt{L}^2(\R)$,
given by
\begin{equation}\label{operatorAi}
G_i(t)\psi= -\alpha_i(t)\,\psi^{\prime\prime} + \beta_i(t)\,\psi^\prime,\;\;\psi \in D\left({G_i}(t)\right)= \mathtt{H}^2(\R).
\end{equation}

The solutions discussed herein are weak solutions (mild solutions), as defined
below:
\begin{definition}
Given  $\phi = (\phi_1, \ldots, \phi_n) \in \mathtt{L}^2(\R)^n$, a local solution to problem \eqref{combustion-problem3} 
is a function
$u = (u_1,\ldots, u_n)$ $\in C\big([0,\,T], \,\mathtt{L}^2(\R)^n \big)$
that satisfies the following integral equation:
\begin{align}\label{integral-form1}
u(t) = U(t,\,0) \phi + \int_0^t\,U(t,\,\tau)\,f\big(\tau,\,u(\tau)\big)\,d\tau, \,\,\,\, t \in [0, T],
\end{align}
for some $T > 0$, where
$U$ is the evolution operator associated with $G(t)$, defined by 
Lemma~\ref{existence-U}. The solution is global in time if it satisfies \eqref{integral-form1} for any $T > 0$. 
\end{definition}

Our main results are summarized in the next three theorems,
whose proofs utilize a novel approach to combustion in porous
media, namely, the semigroups theory of operators in a Hilbert
space and Kato's theory for the initial value problem.

\begin{theorem} [Local solution]\label{combustion-local-theo} 
	Assuming that the hypotheses $(H1)$, $(H2)$, and $(H3)$ given in Section~\ref{existence-U-section} are satisfied, then
	the initial value problem \eqref{combustion-problem3} has a unique mild solution denoted by $u = (u_1, \ldots, u_n) \in C\big([0,\,T], \,\mathtt{L}^2(\R)^n \big)$, for some $T > 0$, provided that $\phi = (\phi_1, \ldots, \phi_n) \in \mathtt{L}^2(\R)^n$.
\end{theorem}

\begin{theorem}[Global Solution] \label{combustion-global-theo} 
	Assuming that the hypotheses $(H1)$, $(H2)$, and $(H3)$ given in
	Section~\ref{existence-U-section} are satisfied, where
	$\Omega_T = \R \times [0, T]$ is replaced by $\Omega = \R
	\times [0, \infty)$ in $(H2)$, then the initial value problem
	\eqref{combustion-problem3} has a global unique mild solution
	denoted by $u = (u_1, \ldots, u_n) \in C\big([0, \infty),
	\mathtt{L}^2(\R)^n \big)$, provided that $\phi = (\phi_1, \ldots,
	\phi_n) \in \mathtt{L}^2(\R)^n$. 
\end{theorem}

\begin{theorem}[Continuous dependence] \label{combustion-continuous-dependence-theo}
Under the same hypotheses as in Theorem~\ref{combustion-local-theo},
then the corresponding solution given by 
this theorem is continuous in the $\mathtt{L}^2(\R)^n\text{-norm}$ with respect to the initial data and parameters.	
\end{theorem}

We note that the proof of Theorem~\ref{combustion-local-theo}
in Section~\ref{local-solution} applies to any fixed $T>0$. Thus,
by replacing $\Omega_T$ with $\Omega$ in hypothesis $(H2)$, the
proof of Theorem~\ref{combustion-continuous-dependence-theo}
covers both local and global solutions.
\section{Preliminaries}\label{preliminaries}
In this section, we present some results about the abstract linear evolution operator that will be used in the current study.

\subsection{Notations and general definitions}\label{notation}
The real numbers set is denoted by $\R$, $I \subset \R$ is an interval, and $T$ is a positive number. We denote by $X$ and $Y$ Banach spaces, where $Y \subset X$. 
The $\mathtt{L}^p$ spaces utilized
in this study include $\mathtt{L}^1$, $\mathtt{L}^2$ and
$\mathtt{L}^{\infty}$. Additionally, the Sobolev spaces
$\mathtt{H}^2$ are used throughout this work. Furthermore, the
following notations are adopted:
\begin{itemize}
	%\item[] $| \cdot |$ is the absolute value or Euclidean norm in $\R^n$. 
	\item[] $\| \cdot \|_X$ is the norm in $X$. 
	
	If there is no ambiguity both norms 
	$\| \cdot \|_{\mathtt{L}^2(\R)}$ and $\| \cdot \|_{\mathtt{L}^2(\R)^n}$ will be denoted by $\| \cdot \|$, where for $\psi=(\psi_1,\ldots,\psi_n) \in \mathtt{L}^2(\R)^n$, $\|\psi\|_{\mathtt{L}^2(\R)^n}=\max_{1\leq i \leq n}\|\psi_i\|_{\mathtt{L}^2(\R)}$.
	\item[] $\partial_{x}=\frac{\partial}{\partial x},\,\partial_{t}=\frac{\partial}{\partial t}$.
	\item[] $\Omega=\{(x,t);\, x \in \R,\,\,t \ge 0\}$.
	\item[] $\Omega_T=\{(x,t);\, x \in \R,\,\,0 \leq t \leq T, \,\,\, T > 0\}$.
	\item[] $C(I,\, X)$ is the space of continuous functions defined from $I$ into  $X$. 
	If $I$ is compact, then it is a Banach space with the supremum norm.
	\item[] $d(u,v)=\sup_{[0,T]}\|u(t)-v(t)\|_X$ is the metric in $C(I,\, X)$.
	\item[]   $C^{2,1}(\Omega_T)$ is the space of twice continuously differentiable functions in $x$ and
	continuously differentiable in $t$.
	\item[] $\mathcal{B}(Y,X)$ \big($\mathcal{B}(X) = \mathcal{B}(X,X)$\big) is the space of all bounded linear operators from $Y$ to $X$ with norm $\|\cdot\|_{\mathcal{B}(Y,X)}$.
\end{itemize}

\subsection{Existence of the evolution operator}\label{existence-U-section}
The following hypothesis is necessary to ensure the existence of the evolution operator associated with the operator $G(t)$.

\begin{itemize}\label{H1} 
	\item[] $(H1)$: The functions $a_i$, $b_i$, $c_i$, and $\lambda_i$ are twice differentiable,
	following conditions: 
	\begin{itemize} 
		\item [(i)] For all $x\in \R$,
		$a_i(x)$ and $\lambda_i(x)$ belong to the interval $[k_1, k_2]$,
		and $b_i(x)$ and $c_i(x)$ belong to the interval $[0, k_2]$,
		where $k_1$ and $k_2$ are positive constants such that $k_1 < k_2$. 
		\item [(ii)] The derivatives $a_i^{(k)}$, $b_i^{(k)}$,
		$c_i^{(k)}$, and $\lambda_i^{(k)}$ for $k= 0,  1,  2$, belong to the space $\mathtt{L}^\infty(\R)$. 
	\end{itemize} 
    \end{itemize}

   Note that if $a_i$, $b_i$, 	$c_i$, and $\lambda_i$ are constants, then all of these
	hypotheses are satisfied.
\begin{itemize} \label{H1}
	\item[]$(H2)$: 
	The function $y_i$ is non-negative and satisfies the following properties: 
	\begin{itemize} 
		\item [(i)] It is twice differentiable with respect to $x$ and differentiable with respect
		to $t$ for all $(x, t) \in \Omega_T$. Moreover, $y_i$, ${(y_i)}_x$,
		${(y_i)}_{xx}$, and $ {(y_i)}_t$ belong to $\mathtt{L}^\infty(\Omega_T)$, and $\|y_i\|_{\mathtt{L}^\infty(\Omega_T)} \leq k_3$, where $k_3$ is a positive constant.
		\item [(ii)] $t \mapsto {(y_i)}_t$ is integrable on $[0, T]$ for all
		$x \in \R$. 
		\item [(iii)] ${(y_i)}_t$ is twice differentiable with
		respect to $x$, $ {(y_i)}_{tx} \in \mathtt{L}^\infty(\Omega_T)$
		and $x \mapsto (y_i)_{txx} \in \mathtt{L}^2(\R)$ for all $t \in
		[0, T]$. 
	\end{itemize}
\end{itemize}
	The following lemma was proved in \cite{Alarcon}, so we will omit its proof here.

\begin{lemma}\label{existence-Ui}
	Assuming that $(H1)$ and $(H2)$ are satisfied, a unique family of semigroup evolution operators
	$U_i(t, t')$ associated with $G_i$ exists and is defined in a triangular domain as follows:
	$$
	(t,\,t^\prime) \in \triangle = \left\{(t,\,t^\prime) \in \R^2 \colon 0 \leq t^\prime \leq t \leq
	T\right\}\longmapsto U_i(t,\,t^\prime) \in \mathcal{B}\left(\mathtt{L}^2(\R)\right) .
	$$
	The operators satisfy the following properties:
	
	\begin{itemize}
		\item [(i)] $(t,\,t^\prime) \longmapsto U_i(t,\,t^\prime) \in \mathcal{B}\left(\mathtt{L}^2(\R)\right)$ is strongly continuous
		and $U_i(t,\,t) = I$ for all $t \in [0,\,T]$;
		\item [(ii)] $U_i(t,\,t^{\prime\prime})= U_i(t,\,t^\prime)\,U_i(t^\prime,\,t^{\prime\prime})$ for all $t$, $t^\prime$,
		$t^{\prime\prime}$ such that $ 0 \leq t^{\prime\prime} \leq t^\prime \leq t \leq T$;
		\item [(iii)] $U_i(t,\,t^\prime)\left(\mathtt{H}^2(\R)\right) \subset \mathtt{H}^2(\R)$ and $(t,\,t^\prime) \longmapsto
		U_i(t,\,t^\prime) \in \mathcal{B}\left(\mathtt{H}^2(\R)\right)$ is strongly continuous in $\mathtt{H}^2(\R)$.
    \end{itemize}
\end{lemma}

Using Lemma \ref{existence-Ui}, it can describe the evolution operator associated with $G(t)$, as in the following lemma:

\begin{lemma}\label{existence-U}
	Assuming that $(H1)$ and $(H2)$ are satisfied, there exists a unique family of 
	evolution operators $U(t, t')$ associated with $G$ such that
	\begin{align*} 
	(t,\,t^\prime) \in \triangle & = \left\{(t,\,t^\prime) \in \R^2 \colon 0 \leq t^\prime \leq t \leq
	T\right\}\longmapsto U(t,\,t^\prime) \in \mathcal{B}(\mathtt{L}^2(\R)^n),
	\end{align*}
where $U(t,t')=(U_1(t,t'),\ldots,U_n(t,t'))$ and $U_i(t,t')$, for $i = 1, \ldots, n$, are given by Lemma \ref{existence-Ui}.
	The operator $U(t,\,t^\prime)$ satisfy the following properties:
	\begin{itemize}
		\item [(i)] $(t,\,t^\prime) \longmapsto U(t,\,t^\prime) \in \mathcal{B}(\mathtt{L}^2(\R)^n)$ is strongly continuous
		and $U(t,\,t) = I$ for all $t \in [0,\,T]$;
		\item [(ii)] $U(t,\,t^{\prime\prime})= U(t,\,t^\prime)\,U(t^\prime,\,t^{\prime\prime})$ for all $t$, $t^\prime$,
		$t^{\prime\prime}$ such that $ 0 \leq t^{\prime\prime} \leq t^\prime \leq t \leq T$.
		\item [(iii)] There exists $\beta>0$ such that 
		$$\|U(t,t')\psi\|\leq e^{\beta t}\|\psi\|,$$
		where $\psi \in \mathtt{L}^2(\R)^n$ and $t,t'\in \triangle.$
	\end{itemize}
\end{lemma}			 
The proof of the last Lemma can be founded in \cite{Alarcon}.
Then, the mild solution to problem \eqref{combustion-problem3} is given  by the following equation:

\begin{align}\label{Aintegral-form1}
u(t) = U(t,\,0) \phi + \int_0^t\,U(t,\,\tau)\,f\big(\tau,\,u(\tau)\big)\,d\tau, \,\,\,\, t \in [0, T].
\end{align}

\section{Local Solution}\label{local-solution}
In this section, we will describe the technique employed to obtain the local solution. Here, we will adopt the ideas of Kato, which can be found in \cite{Kato4} and \cite{Kato5}. See also \cite{Cunha}.

Next lemma shows that the source function $f$ satisfies the necessary properties for the existence of the local solution. Its proof requires the following additional hypothesis:

\begin{itemize}\label{H3}
\item[] $(H3)$: The functions $d_i$, $q_i$, $K_i$, are in $\mathtt{L}^\infty(\R)$, and $\hat{q}_1$, $\hat{q}_2$, are in $\mathtt{L}^2(\R)$ . 
\end{itemize}

\begin{lemma} \label{properties-fi} 
	Let $(H1)$, $(H2)$, and $(H3)$ be satisfied. Given an open ball $W \subset \mathtt{L}^2(\R)^n$, centered at origin, with  radius $0 <\rho <  \infty$, 
	then the source function $f = (f_1, \ldots, f_n):[0, T] \times W \to \mathtt{L}^2(\R)^n$, for any fixed $T >0$, satisfies the following properties: 
	\begin{itemize} 
		\item[(i)]  There exists a constant $\mu>0$, which depends on the radius of $W$ but not on $t$, such that $\|f(t, \,w)\|\leq \mu $ for all 
		$t \in [0,\,T]$ and $w \in W$.
		\item[(ii)] For a fixed $w \in W$, the function $t\in [0, T] \mapsto f(t, \,w)\in \mathtt{L}^2(\R)^n$ is continuous.
		\item[(iii)] For each $t \in [0,\,T]$, the function $w \mapsto f(t, \,w)$ is Lipschitz in $\mathtt{L}^2(\R)^n$, meaning that:
		$$
		\| f(t, \,v) - f(t, \,w)\| \leq \kappa\, \| v - w\|, 
		$$ 
		for all  $v,w \in \mathtt{L}^2(\R)^n$, where the Lipschitz constant $\kappa$ 
		does not depend on $t$.
	\end{itemize}
\end{lemma}	

\begin{proof}
We will only prove the property (iii) here. The proof of (i) is similar to the proof of (iii). The proof of (ii) is a consequence of the  continuity of $y_i$ and the definition of $f_i$. 

Proof of (iii): We will give the proof for $f_1$, for 
$f_2,\ldots, f_n$ the proofs follow in a similar way. 
It is sufficient to prove that, for each $t \in [0, T]$, the function
$w \mapsto f_1(t, \,w)$  is Lipschitz continuous, that is,
$$\| f_1(t, \,v) - f_1(t_, \,w)\| \leq \kappa_1\, \| v - w\|,$$
for all  $w = (w_1, \dots, w_n)$, $v = (v_1, \dots, v_n) \in \mathtt{L}^2(\R)^n$, where the Lipschitz constant $\kappa_1$ does not depend on $t$. In this case, $\kappa:= \max\{\kappa_i, \,\, i= 1, \ldots, n \}$. %which depends on $k_1$ and $\tilde R$.

We define the following functions:
$\gamma(x,t)=\frac{(c_1)_x}{a_1+b_1y_1}$,
$\delta(x,t)=\frac{K_1 b_1 y_1}{a_1+b_1 y_1}$,
$\sigma(x,t)=\frac{d_1 y_1}{a_1+b_1 y_1}$,
$\epsilon(x,t)=\frac{q_1}{a_1+b_1 y_1}$, and $\omega(x,t)=-
\frac{\hat{q}_1}{a_1+b_1y_1}$. 

From hypotheses $(H1)$, $(H2)$, and $(H3)$ we have 
$$\|\gamma\|_{\mathtt{L}^\infty}, \; \|\delta\|_{\mathtt{L}^\infty},  \; \|\sigma\|_{\mathtt{L}^\infty},  \; \|\epsilon\|_{\mathtt{L}^\infty},  \; 
\|\omega\|_{\mathtt{L}^\infty} \leq const,$$ 
where $const$ depends on $k_1, \, k_2, \, \|d_1\|_{\mathtt{L}^\infty},\, \|q_1\|_{\mathtt{L}^\infty}, \, \|K_1\|_{\mathtt{L}^\infty}, \, 
\|(c_1)_x\|_{\mathtt{L}^\infty}$ and \break $\|\hat{q}_{1}\|_{_{\mathtt{L}^\infty}}$ but does not depend on $t$.

Thus, using the definition of $f_1$ given in \eqref{reaction-function-1}, we have
\begin{equation}
	\begin{split}\label{f1uv}
		f_1(x,t,v)-f_1(x,t,w)=& \; \gamma (v_1-w_1)+\delta \big(g(v_1)v_1 - g(w_1)w_1\big)\\
		&+\sigma (g(v_1)-g(w_1))+\epsilon (v_2-w_2)\\
		&+(\epsilon+\omega)(w_1-v_1).                     
	\end{split}
\end{equation}
For each $t\in [0,T]$, from the definition of $g$ (see \eqref{darcy}) and Mean value theorem it follows that there exists $\theta$ between $v_1$ and $w_1$ such that

\begin{equation*} %\label{gu1}
	\|g(v_1)v_1 - g(w_1)w_1\| = \|(g'(\theta)\theta + g(\theta))(v_1 - w_1)\|. \quad 
	%\text{where} \quad v_1 < \theta < w_1,	
\end{equation*}

Then, we have
\begin{equation} \label{gu1}
\|g(v_1)v_1 - g(w_1)w_1\| \leq const_1\, \|v_1 - w_1\|,
\end{equation}
where $const_1$ is a constant independent on $t$.
Hence, \eqref{f1uv} and \eqref{gu1} imply that

\begin{equation}
	\begin{split}\label{f1uvest}
		\|f_1(x,t,v)&-f_1(x,t,w)\| \leq \; \|\gamma\|_{\mathtt{L}^\infty} \|v_1-w_1\| +\\
		&+ \|\delta\|_{\mathtt{L}^\infty} const_1\, \|v_1 - w_1\| \\
		&+\|\sigma\|_{\mathtt{L}^\infty} \|g'\|_{\mathtt{L}^\infty} \|v_1 - w_1\| +\|\epsilon\|_{\mathtt{L}^\infty} \|v_2 - w_2\| \\
		&+(\|\epsilon\|_{\mathtt{L}^\infty}+\|\omega\|_{\mathtt{L}^\infty})\|v_1 - w_1\| \\
		\leq & \: \underbrace{\big((\rho\|g'\|_{\mathtt{L}^\infty}+\|g\|_{\mathtt{L}^\infty})+c_1 \|g'\|_{\mathtt{L}^\infty}+4c_1\big)}_{\kappa_1}\|v-w\|,         
	\end{split}
\end{equation}
where $\kappa_1$ is independent on $t$, which completes the proof.

\end{proof}

\subsection{Proof of Theorem \ref{combustion-local-theo}}\label{proof-teo-local} 
\begin{proof}	
Assume $T > 0$, and $W \subset \mathtt{L}^2(\R)^n $ be an open ball centered at the origin, as defined in Lemma~\ref{properties-fi}, such that $\phi$ belongs to $W$. 

Now, consider a number $R > \rho e^{\beta T}$, and define an open ball $W' \subset \mathtt{L}^2(\R)^n $ centered at the origin, with radius $R$. It is worth noting that $W'$ includes $W$.  

Let us consider the space
\begin{align} \label{ET}
	E_{T}=\{u\in C([0, T], \,\mathtt{L}^2(\R)^n): \,\, &\sup_{[0,T]}\|u(t)\|\leq R\},
\end{align}
which is not empty, because by Lemma \ref{existence-U} (iii), $g(t)\equiv U(t,0)\phi \in E_{T} $.

	It is not difficult to prove that $E_{T}$ is a complete metric space with 
\begin{equation*}
	d(u,v)=\sup_{[0,T]}\|u(t)-v(t)\|. 
\end{equation*}

We define the function
\begin{equation}\label{function-contraction}
\Phi u (t)=U(t,0)\phi+\int_0^t  U(t,\tau)f(u(\tau))d\tau, \quad t\in [0,T],
\end{equation}
and our goal is to prove that there exists $0< T' \leq T$ such that $\Phi$ is a contraction in $E_{T'}$.

First, if $u\in E_{T}$, from Lemmas \ref{existence-U}, \ref{properties-fi}, and Lebesgue dominated convergence 
theorem, we conclude that $\Phi u \in C([0, T], \,\mathtt{L}^2(\R)^n)$.
Taking the $\mathtt{L}^2(\R)^n$-norm in \eqref{function-contraction}, and using Lemmas 
\ref{existence-U} (iii), and \ref{properties-fi} (i), for all  $0 \leq t  \leq T'$, we obtain

\begin{equation}\label{contE2}
\begin{split}
\|\Phi u (t)\| \leq e^{\beta T}\|\phi\| + e^{\beta T}\int_0^t \|f(u(\tau))\| d\tau \leq e^{\beta T}(\rho+T'\mu). 
\end{split}
\end{equation}
We also have that,
\begin{equation}\label{cont}
\begin{split}
\|\Phi u (t)-\Phi v (t)\| &\leq \int_0^t \|U(t,\tau)(f(u(\tau))-
f(v(\tau)))\| d\tau\\
&\leq e^{\beta T} \int_0^t \|f(\cdot,\tau,u(\tau))-f(\cdot,\tau,v(\tau))\| d\tau\\
& \leq  \kappa e^{\beta T}\int_0^t \|u(\tau)-v(\tau)\| d\tau\\
& \leq T' \kappa e^{\beta T} d(u,v).
\end{split}
\end{equation}
Now, by choosing 
\begin{equation}
0<T'<\min\Big \{T, \frac{1}{\kappa e^{\beta T}}, \frac{1}{\mu}\Big(\frac{R}{e^{\beta T}}-\rho\Big)\Big\},
\end{equation}  
it follows from inequalities  \eqref{contE2} 
and \eqref{cont} that $\Phi : E_{T'}\to E_{T'}$ is a contraction. Therefore, from the Banach fixed-point theorem, the initial value problem \eqref{combustion-problem3} has a unique mild solution $u$ with $u(0) = \phi$, which can be expressed as \eqref{Aintegral-form1}. 

It remains to prove the uniqueness.
	
	Let $u$ and $v$ be solutions of \eqref{Aintegral-form1} such that $u(0)=v(0)=\phi$. For any $t\in [0,T']$, it follows from  Lemmas \ref{existence-U} (iii) and \ref{properties-fi} (iii) that:
	\begin{equation*}
		\begin{split}
			\|u(t)-v(t)\| &\leq \int_0^t \|U(t,\tau)(f(u(\tau))-f(v(\tau)))\| d\tau\\
			&\leq e^{\beta T'}\int_0^t \|f(\cdot,\tau,u(\tau))-f(\cdot,\tau,v(\tau))\| d\tau\\
			&\leq \kappa e^{\beta T'}\int_0^t \|u(\tau)-v(\tau)\| d\tau.
		\end{split}
	\end{equation*} 
	Therefore, according to Gronwall's lemma $u(t)=v(t)$, which completes the proof of uniqueness.
\end{proof}

\section{Global Solution}\label{global-solution}

This section aims to prove Theorem \ref{combustion-global-theo}.  The arguments presented here are inspired by the theory of
global existence contained in \cite{pazy}. First, we consider the  auxiliary problem 
\begin{equation}\label{Acomb-globalt0}
\left\{
\begin{array}{l}
u_t + G(t)u = f(t, u), \,\,\, t>t_0, \\
u(t_0) = \phi,\\
\end{array}
\right.
\end{equation}
for $t_0\geq 0$.

We will require the following result.

\begin{lemma}\label{global-pazy}
If $(H1)$, $(H2)$, and $(H3)$ are satisfied, where $\Omega_T$ is replaced by $\Omega$ in $(H2)$, and $\phi \in \mathtt{L}^2(\R)^n$, then there exists $t_{\mathrm{max}}\leq \infty$ such that the initial value problem~ \eqref{combustion-problem3},
%\begin{equation}\label{Acomb-global}
%\left\{
%\begin{array}{l}
%u_t + G(t)u = f(t, u), \,\,\, t>0, \\
%u(0) = \phi,\\
%\end{array}
%\right.
%\end{equation}
has a unique mild solution $u$ defined on $[0,t_{\mathrm{max}})$.

Furthermore, if $t_{\mathrm{max}}<\infty$ then $$\lim_{t\uparrow t_{\mathrm{max}}}\|u(t)\|=\infty.$$

\end{lemma}
\begin{proof}
First, we will show that for all $t_0\geq 0$ and $\phi \in \mathtt{L}^2(\R)^n$, the initial value problem \eqref{Acomb-globalt0} has a unique mild solution $u$ on an interval $[t_0,t_1]$. To do this, take $\kappa$ and $\mu$ as in Lemma \ref{properties-fi}, then we choose  %and%, $M(t_0)=\max\{\|f(\tau,0)\|: t_0\leq \tau \leq t_0 +1\}$ and

\begin{equation}\label{defepsilon}
\epsilon=\epsilon(t_0,\|\phi\|):=\min \Bigg\{1,\dfrac{\|\phi\|}{\kappa\;R(t_0)+\mu}\Bigg\},
\end{equation}
where $R(t_0)=2\|\phi\|e^{\beta(t_0 +1)}$.

Let $t_1=t_0+\epsilon$ and $u\in C([t_0,t_1], \,\mathtt{L}^2(\R)^n)$ such that $\|u(t)\|\leq R(t_0)$, for every $t\in [t_0,t_1]$.

Defining the function
\begin{equation}\label{intglobal}
\Psi  u \;(t)=U(t,t_0)\phi +\int_{t_0}^t  U(t,\tau)f(\tau, u(\tau))d\tau, \quad \tau \in [t_0,t_1],
\end{equation}
from Lemmas \ref{existence-U} (iii) and \ref{properties-fi} (i--iii), we have that

\begin{equation}
\begin{split}
\|\Psi  u (t)\|&\leq \; \|U(t,t_0)\phi\|+\int_{t_0}^t  \|U(t,\tau)f(\tau, u(\tau))\|d\tau\\
           &\leq \; \|\phi\|e^{\beta(t_0 +1)}+e^{\beta(t_0 +1)}\int_{t_0}^t \|f(\tau, u(\tau))\|d\tau\\
           &\leq \; e^{\beta(t_0 +1)}\Bigg[\|\phi\|+\int_{t_0}^t \big(\|f(\tau, u(\tau))-f(\tau,0)\|+\|f(\tau,0)\|\big)d\tau\Bigg]\\
           &\leq \; e^{\beta(t_0 +1)}\Big[\|\phi\|+(t-t_0)\big(\kappa\; R(t_0)+\mu\big)\Big]\\
           &\leq \; e^{\beta(t_0 +1)}\Big[\|\phi\|+\epsilon\big(\kappa\; R(t_0)+\mu\big)\Big]\\
           &\leq \; 2\|\phi\|e^{\beta(t_0 +1)} =\; R(t_0).
\end{split}
\end{equation}
Thus, we have $\Psi(B) \subset B$ , where $B$ denotes the closed ball with radius $R(t_0)$ centered at $0$ and contained in $C([t_0,t_1], \,\mathtt{L}^2(\R)^n)$. 

As a consequence of the proof of Theorem~\ref{combustion-local-theo}, we can deduce that $\Psi$ possesses a unique fixed point $u$ in $B$, which represents the solution of \eqref{Acomb-globalt0} on $[t_0,t_1]$.

Suppose that $u$ is a mild solution of \eqref{combustion-problem3}
defined on $[0,s]$. Based on the aforementioned result, we can
conclude that the integral equation 
\begin{equation*}
	w(t)=U(t,s)u(s) +\int_{s}^t U(t,\tau)f(\tau, w(\tau))d\tau
\end{equation*} 
has a solution $w$ that is defined on
$[s,s+\epsilon]$, where $\epsilon>0$ depends only on $\|u(s)\|$,
$R(s)$, $\kappa$ and $\mu$. Thus, we can extend $u$ to the interval
$[0,s+\epsilon]$ by defining $u(t)=w(t)$ on $[s,s+\epsilon]$.

Consider $[0,t_{\mathrm{max}})$ the maximal interval of existence of the mild solution $u$ to \eqref{combustion-problem3}. Assuming $t_{\mathrm{max}}<\infty$ it follows that $\lim_{t\uparrow t_{\mathrm{max}}}\|u(t)\|=\infty$. Indeed, if not there is a sequence $t_m\uparrow t_{\mathrm{max}}$ such that  $\|u(t_m)\|\leq C$. Then using \eqref{defepsilon} we set 

\begin{equation*}
\epsilon_m:=\min \Bigg\{1,\dfrac{\|u(t_m)\|}{\kappa\;R(t_m)+\mu}\Bigg\},
\end{equation*}

thus, we see that there exists $\epsilon:=\lim_{m\to \infty}\epsilon_m$. To conclude this, it's enough to observe that the triangle inequality and integral equation \eqref{Aintegral-form1} imply that $\{\|u(t_m)\|\}$ is a Cauchy sequence. On the other hand, there exists $N>0$ such that $t_{\mathrm{max}}-t_m<\epsilon$, for $m>N$. Thus, the solution $u$ defined on $[0,t_m]$ can be extended to $[0,t_m + \epsilon]$. Therefore, $u$ can be extended beyond $t_{\mathrm{max}}$, which contradicts the definition of $t_{\mathrm{max}}$.

The uniqueness follows as in the proof of Theorem~\ref{combustion-local-theo}.
\end{proof}

\subsection{Proof of Theorem \ref{combustion-global-theo}}\label{proof-teo-global}
\begin{proof}
Let be $u$ the mild solution of \eqref{combustion-problem3} defined on the interval $[0,T]$, as given by 
Theorem~\ref{combustion-local-theo}. Based on Lemmas \ref{existence-U} (iii) and \ref{properties-fi} (i--iii), we can deduce that

\begin{equation}\label{priori-est}
\begin{split}
\|u(t)\|&\leq \|U(t,0)\phi\|+\int_0^t \|U(t,\tau)f(\tau,u(\tau))\|d\tau\\
        &\leq \; e^{\beta t}\Bigg(\|\phi\|+\int_0^t \|f(\tau,u(\tau))\|d\tau\Bigg)\\
        &\leq \; e^{\beta t}\Bigg(\|\phi\|+\int_0^t \big(\|f(\tau,0)\|+\|f(\tau, u(\tau))-f(\tau,0)\|\big)d\tau\Bigg)\\
        &\leq \; e^{\beta t}\Bigg(\|\phi\|+\mu t+\kappa \int_0^t \|u(\tau)\|d\tau\Bigg), \quad t\in [0,T].
        %&\leq \; e^{\beta t}(\|\phi\|+\mu t), \quad t\in [0,T].
\end{split}
\end{equation}	

Using \eqref{priori-est}, Gronwall's lemma, and Lemma~\ref{global-pazy}, we conclude that $t_{\mathrm{max}}=\infty$, where $[0,t_{\mathrm{max}})$ is the maximal interval of existence of the solution $u$, which concludes the proof. 
\end{proof}

\section{Continuous dependence}

In this section, we examine the continuous dependence of the
solution to problem \eqref{combustion-problem3} concerning the
initial data and parameters together. As mentioned in
Section~\ref{introduction}, since the proof presented here is valid
for any fixed $T>0$, it encompasses both local and global
solutions.

Next, we consider a sequence of problems similar to
\eqref{combustion-problem3} as follows:

\begin{equation}\label{combustion-problem3j}
\left\{\begin{array}{l}
\partial_t u+G^j(t)u=f^j(t,u),\\
u^j(0)=\phi^j .
\end{array}\right.
\end{equation}
Here, $G^j(t) := \big(G_{1}^j(t), \ldots, G_{n}^j(t) \big)$, where
$G_{i}^j(t)$ acts on the function $\psi \in D\left(G_{i}^j(t)\right)=
\mathtt{H}^2(\R)$ according to the following definition:
\begin{align}\label{operatorAij}
G_{i}^j(t)&\psi = -\alpha_{i}^j(t)\,\psi^{\prime\prime} + \beta_{i}^j(t)\,\psi^\prime, 
\end{align}
where
\begin{align}
\label{coeficients_appha_beta}
& \alpha_{i}^j(t)=\dfrac{(\lambda_i)^j}{(a_i)^j+(b_i)^j(y_i)^j(t)}\,\,  \mbox{and} \,\, \beta_{i}^j(t)=\dfrac{(c_i)^j}{(a_i)^j+(b_i)^j(y_i)^j(t)}.
\end{align}
For every natural $j$, the functions $(a_i)^j, \,\, (b_i)^j, \,\, (c_i)^j,  \,\,(\lambda_i)^j : \R \rightarrow \R$  satisfy $(H1)$, and $(y_i)^j :\Omega_T \rightarrow \R$ satisfies $(H2)$.
The initial data $\phi^j = \big((\phi_1)^j, \ldots, (\phi_n)^j \big)$ is a given function.
%taken in the open ball $W$ defined by Lemma~\ref{properties-fi}. 
Finally, the source functions $f^j = \big((f_1)^j, \ldots, (f_n)^j \big)$ are defined by
\begin{align} \label{raction-function-j} 
		&(f_1)^j(t, w) = \frac{-((c_1)^j)_x\,w_1 }{(a_1)^j + (b_1)^j (y_1)^j(t)} + 
		\frac{\big((K_1)^j (b_1)^j w_1+(d_1)^j\big)(y_1)^j(t)\,g(w_1)}{(a_1)^j + (b_1)^j (y_1)^j(t)} 
		\nonumber \\
		& \phantom{------}+ \frac{(q_1)^j(w_2 - w_1)}{(a_1)^j + (b_1)^j (y_1)^j(t)}
		- \frac{(\hat q_1)^j(w_1 - u_e)}{(a_1)^j + (b_1)^j (y_1)^j(t)}\,,
		\nonumber \\
		\nonumber \\
		&(f_i)^j(t, w) = \frac{-((c_i)_x)^j\,w_i }{(a_i)^j + (b_i)^j (y_i)^j(t)} + 
		\frac{\big((K_i)^j (b_i)^j w_i+(d_i)^j\big)(y_i)^j(t)\,g(w_i)}{(a_i)^j + (b_i)^j (y_i)^j(t)} 
		\nonumber \\
		& \phantom{----}+ \frac{(q_i)^j(w_{i+1} - w_i)}{(a_i)^j + (b_i)^j (y_i)^j(t)}
		- \frac{(q_{i-1})^j(w_i - w_{i-1})}{(a_i)^j + (b_i)^j (y_i)^j(t)}\,, \,\,i = 2,\ldots, n-1,
		\\
		&(f_n)^j(t, w) = \frac{-((c_n)^j)_x\,w_n }{(a_n)^j + (b_n)^j (y_n)^j(t)} + 
		\frac{\big((K_n)^j (b_n)^j w_n+(d_n)^j\big)(y_n)^j(t)\,g(w_n)}{(a_n)^j + (b_n)^j (y_n)^j(t)} 
		\nonumber \\
		& \phantom{------} - \frac{(q_{n-1})^j(w_n - w_{n-1})}{(a_n)^j + (b_n)^j (y_n)^j(t)}
		- \frac{(\hat q_2)^j(w_n - u_e)}{(a_n)^j + (b_n)^j (y_n)^j(t)}\,,
		\nonumber 
\end{align}
where for each natural $j$, the functions $(d_i)^j, \,\, (q_i)^j, \,\, (K_i)^j,  \,\,(\hat{q}_1)^j, \,\, (\hat{q}_2)^j : \R \rightarrow \R$ satisfy $(H3)$. 

Let $T > 0$ and $W \subset \mathtt{L}^2(\R)^n$ be an open ball centered at the origin such that $\phi$ belongs to $W$. 
Assuming that hypotheses $(H1)$, $(H2)$, and $(H3)$ are satisfied, we proved in 
Section~\ref{proof-teo-local} that the problem \eqref{combustion-problem3} has an unique solution $u \in C\big([0,\,T], \,W' \big)$, where $W'$ is an open ball in $\mathtt{L}^2(\R)^n$, centered at the origin, containing $W$. 

\begin{lemma}\label{lema1} For each natural $j$, we assume that the functions $(a_i)^j, (b_i)^j, \break(c_i)^j$, $(\lambda_i)^j$ satisfy $(H1)$, $(y_i)^j(.,t)$ satisfy $(H2)$ for each $t\in[0,T]$, and the functions $(d_i)^j, (q_i)^j, (K_i)^j$, $(\hat{q}_1)^j$, $(\hat{q}_2)^j$ satisfy $(H3)$.  
	
	We also assume that $(a_i)^j\rightarrow a_i$, $(b_i)^j\rightarrow
	b_i$, $(c_i)^j\rightarrow c_i$, $((c_i)^j)_x \rightarrow (c_i)_x$, $(\lambda_i)^j\rightarrow
	\lambda_i$, $(y_i(.,t))^j\rightarrow y_i(.,t)$ for each $t \in [0, T]$, $(d_i)^j\rightarrow d_i$, $(q_i)^j\rightarrow q_i$, $(K_i)^j\rightarrow K_i$, $(\hat q_1)^j\rightarrow \hat q_1$, and $(\hat q_2)^j\rightarrow \hat q_2$, where all convergences are in $\mathtt{L}^\infty(\R)$. Then, we have:
	\begin{enumerate} 
		\item[(i)] $G_{i}^j(t)\rightarrow G_i(t)$ in
		$\mathcal{B}(\mathtt{H}^2,\mathtt{L}^2)$; 
		\item[(ii)] $\lim_{|E|\rightarrow 0}\int_E
		||G_{i}^j(t)||_{\mathcal{B}(\mathtt{H}^2,\mathtt{L}^2)}dt\rightarrow 0$; 
		\item[(iii)] For each $u\in E_T$, $f^j(t,u(t))\rightarrow f(t,u(t))$ in $\mathtt{L}^2(\R)$,
		pointwise in $t$, where $f_i(t,u(t)):=f_i(x, t, u(t))$; $(f_i)^j(t,u(t)):=(f_i)^j(x, t, u(t))$, and $E_T$ is defined in
		\eqref{ET}. 
	\end{enumerate} 
\end{lemma}

\begin{proof}
	The proofs of (i) and (ii) can be made similar to the proof in \cite[Lemma~5.3]{Alarcon}.
	For the proof of (i), let's consider $i=2,\dots,n-1$; the cases $i=1$ and $i=n$ are similar. 	
	If $t\in[0,T]$, 
	\begin{equation}
	\begin{split}
		\|(f_i)^j(t,u(t))-&f_i(t,u(t))\|\\ 
		\leq&\left\|\frac{- (c_i)_x }{a_i + b_i y_i}\, -
		\frac{- ((c_i)^j)_x }{(a_i)^j + (b_i)^j (y_i)^j}\right\|_{\mathtt{L}^\infty} \left\|u_i\right\|\\
		+&
		\left\|\frac{b_i K_i y_i}{a_i + b_i y_i}  -\frac{(b_i)^j (K_i)^j (y_i)^j}{(a_i)^j + (b_i)^j (y_i)^j} \right\|_{\mathtt{L}^\infty}\left\|g(u_i) u_i \right\|\\
		+& \left\|\frac{d_i \,y_i }{a_i+b_i y_i}- \frac{(d_i)^j \,(y_i)^j }{(a_i)^j+(b_i)^j (y_i)^j}\right\|_{\mathtt{L}^\infty}\left\|g(u_i)\right\|\\
		+&\left\|\frac{ q_{i-1} }{a_i + b_i y_i}-\frac{ (q_{i-1})^j }{(a_i)^j + (b_i)^j (y_i)^j}\right\|_{\mathtt{L}^\infty} \left\|u_i - u_{i-1}\right\|\\
		+&\left\|\frac{ q_i }{a_i + b_i y_i}-\frac{ (q_i)^j }{(a_i)^j + (b_i)^j (y_i)^j}\right\|_{\mathtt{L}^\infty} \left\|u_{i+1} - u_i\right\|.
		\end{split}
	\end{equation}

	Since $u \in E_T$, the function $g$ is bounded, in fact $\|g(u_i)\|\leq \|g'(\overline{u})u_i\|\leq const\|u_i\|$. Thus $(f_i)^j(t,u(t))\rightarrow f_i(t,u(t))$ in $\mathtt{L}^2(\R)$, independent of $t \in [0, T]$, which completes the proof.
\end{proof}

\begin{corollary} \label{corolf} 
Under the same hypothesis as Lemma \ref{lema1}, for each $u\in E_T$, $f^j$ converges 
to $f$ in $\mathtt{L}^1([0,T], \mathtt{L}^2(\R)^n)$. 
\end{corollary} 
\begin{proof} From Lemma
\ref{lema1}(iii), we know that $f^j(t,u(t))$ converges to
$f(t,u(t))$ in $\mathtt{L}^2(\R)$, pointwise in $t$. Furthermore, there
exists a constant $const>0$ such that $\|f^j(t,u(t))\|\leq const$, for
all $j\in\mathbb{N}$. We conclude by the Lebesgue dominated convergence 
theorem that $\int_0^T \|f^j(t,u(t))-f(t,u(t))\|dt\rightarrow 0$. 
\end{proof}

\begin{lemma}\label{lema2} 
	Assuming the same hypotheses as Lemma \eqref{lema1}, we have that $U^j(t,s)$ $\rightarrow U(t,s)$ strongly in $\mathcal{B}(\mathtt{L}^2(\R)^n)$, uniformly in $t,s\in\Delta$. \end{lemma}
\begin{proof} This is a consequence of $(U_i)^j(t,s)\rightarrow
	U_i(t,s)$ strongly in $\mathcal{B}(\mathtt{L}^2(\R))$, uniformly in $t,s\in\Delta$. The proof of this result is omitted here, as it is similar to the proof in \cite[Theorem V, p.661]{Kato3}. 
\end{proof}

\subsection{Proof of Theorem \ref{combustion-continuous-dependence-theo}}
\label{proof-continuous-dependence}
\begin{proof}
Assuming that $\phi\in W\subset \mathtt{L}^2(\R)^n$ and $\phi^j\rightarrow \phi$ in
$\mathtt{L}^2(\R)^n$, we can consider without loss of generality that $\phi^j \in W$.

Let $u,\, u^j\in C([0,T], \,W)$ be mild local solutions to \eqref{combustion-problem3} and \eqref{combustion-problem3j}, respectively,  with $u(0)=\phi\in W\subset \mathtt{L}^2(\R)^n$ and $u^j(0)=\phi^j\in W\subset \mathtt{L}^2(\R)^n$. Then, we have
\begin{equation*}
\begin{split} 
u^j(t)-u(t)=&\; U^j(t,0)\phi^j-U(t,0)\phi\\
&+\int_0^t U^j(t,\tau)[f^j(\tau,u^j(\tau))-U(t,\tau)f(\tau,u(\tau))]d\tau \\
		=& \; U^j(t,0)(\phi^j-\phi)+(U^j(t,0)-U(t,0))\phi \\
	&+ \int_0^t U^j(t,\tau)[f^j(\tau,u^j(\tau))-f(\tau,u(\tau))]d\tau\\
	&+\int_0^t(U^j(t,\tau)-U(t,\tau))f(\tau,u(\tau))d\tau 	\\
	 =&\; U^j(t,0)(\phi^j-\phi)+(U^j(t,0)-U(t,0))\phi\\
	&+ \int_0^t U^j(t,\tau)[f^j(\tau,u^j(\tau))-f^j(\tau,u(\tau))]d\tau 
		 \\
		 &+ \int_0^tU^j(t,\tau)[f^j(\tau,u(\tau))-f(\tau,u(\tau))]d\tau\\
	& +\int_0^t(U^j(t,\tau)-U(t,\tau))f(\tau,u(\tau))d\tau\\
		 :=&\; \delta_0^j(t)+\delta_1^j(t)+\delta_2^j(t)+\delta_3^j(t)+\delta_4^j(t)\,.
\end{split}
\end{equation*}
	%Therefore,
	%$$\|u^j(t)-u(t)\|\leq \|\delta_0^j(t)\|+ \|\delta_1^j(t)\|+\|\delta_2^j(t)\|+\|\delta_3^j(t)\|+\|\delta_4^j(t)\|.$$
Lemmas \ref{existence-U} and \ref{lema2} imply that 
\begin{equation}\label{consUj}
\|U^j(t,\tau)\|_{\mathcal{B}(\mathtt{L}^2(\R)^n)}\leq 1+ e^{\beta T}:=\;\tilde{c}.
\end{equation}	
	Then we have
	\begin{equation*} 
	\begin{split}
	\|\delta_0^j(t)\|=&\|U^j(t,0)(\phi^j-\phi)\|\leq \; \tilde{c} \;\|\phi^j-\phi\| \to \; 0,\quad \mbox{with}\quad j\rightarrow\infty,
	\end{split}
	\end{equation*}
	since $\phi^j\rightarrow \phi$ in $\mathtt{L}^2(\R)^n$. 

Also, using Lemma \ref{lema2} it follows that
	\begin{equation*} 
		\|\delta_1^j(t)\|=\|(U^j(t,0)-U(t,0))\phi\| \to 0,\quad \mbox{when} \quad j\to\infty.
	\end{equation*}

	From Lemma \ref{properties-fi} (iii) and \eqref{consUj} we obtain
	\begin{equation*} 
	\begin{split}
	\|\delta_2^j(t)\|\leq&\int_0^t\|U^j(t,\tau)(f^j(\tau,u^j(\tau))-f^j(\tau,u(\tau)))\|d\tau\\
		\leq& \; \tilde{c} \int_0^t\|u^j(\tau)-u(\tau)\|d\tau.
		\end{split}
	\end{equation*}
	On the other hand, Corollary \ref{corolf} and \eqref{consUj} yields us
	\begin{equation*} 
	\begin{split}
	\|\delta_3^j(t)\|\leq& \; \int_0^t\|U^j(t,\tau)(f^j(\tau,u(\tau))-f(\tau,u(\tau)))\|d\tau\\
		\leq& \; \tilde{c}\int_0^t\|f^j(\tau,u(\tau))-f(\tau,u(\tau))\|d\tau \to 0,
		\end{split}
	\end{equation*}
	
	and combining the Lebesgue dominated convergence theorem with \break Lemma \ref{lema2} we get	
	\begin{equation*} 
	\|\delta_4^j(t)\|\leq \; \int_0^t\|(U^j(t,\tau)-U(t,\tau))f(\tau,u(\tau))\|d\tau \, \to 0.
	\end{equation*}

	Then, denoting $\delta^j(t)= \|\delta_0^j(t)\|+\|\delta_1^j(t)\|+\|\delta_3^j(t)\|+\|\delta_4^j(t)\|$ and using the information above, we can conclude that
	
	\begin{equation*}
	\|u^j(t)-u(t)\|\leq \;\delta^j(t)+\tilde{c} \int_0^t\|u^j(\tau)-u(\tau)\|d\tau \,.
	\end{equation*}
	
	Thus, an application of Gronwall's inequality gives us
	
	\begin{equation*}
	\begin{split}
		\|u^j(t)-u(t)\|\leq& \; \delta^j(t)+\tilde{c}\int_0^t \delta^j(\tau) e^{\tilde{c} T}d\tau \\
		\leq& \;\sup_{0 \leq t \leq T}\delta^j(t)\left(1+T\tilde{c}e^{\tilde{c}T}\right).
	\end{split}
	\end{equation*}
	Therefore, $u^j$ converge to $u$ in $C([0,T], \,W)$, which completes the proof.
	%\begin{equation*}
	%	\sup_{0\leq t \leq T}\|u^j(t)-u(t)\|\leq \overline{K}\sup_{0 \leq t \leq T}\delta^j(t).
	%\end{equation*}
\end{proof}

\section{Concluding remarks}
The theory of semigroups and Kato's theory were employed to obtain a solution to the initial value problem for a combustion model in a multilayer porous medium.
The approach used proved to be very efficient, allowing for the consideration of a more realistic model where certain physical parameters (e.g., porosity, thermal conductivity, and initial fuel concentration) vary with the spatial variable $x$, instead of being constant as assumed in earlier studies \cite{Batista1, Batista2}. The method enabled the derivation of an explicit formula for the solution through an integral representation, as given by Eq.~\eqref{integral-form1}. The uniqueness of the solution naturally followed from this integral representation, which relies on the uniqueness of the evolution operator both locally and globally in time.
Furthermore, the continuous dependence of the solution with respect to the initial data and the parameters was established due to the integral formula. This is particularly important because practical applications often involve initial data that incorporate rounding errors, making the study of continuous dependence essential for such problems. The results of this work, which considers initial data in $\mathtt{L}^2(\R)^n$ instead of $\mathtt{H}^2(\R)^n$, contribute to the broader goal of analyzing the widest possible space for initial data, both from a theoretical standpoint and in practical applications.

It should be noted that the more realistic problem where fuel concentrations are also unknown functions cannot be directly solved using the semigroup theory, as the system of partial differential equations (Eqs. \eqref{be14}-\eqref{bmf4}) does not fall under the category of reaction-diffusion equations. However, the results obtained in this study can be utilized to construct an iterative sequence that converges to the desired solution.

Several other questions can be explored based on the model described in Appendix~\ref{model}. For instance, if the operator $G$ depends on both $t$ and $u$, the corresponding initial value problem becomes a quasi-linear problem, which is considerably more challenging than the semilinear problem studied here. Additionally, investigating
the local and global well-posedness in $\mathtt{H}^2(\R)^n$ with a bore-like initial condition 
(see \cite{Iorio} for more details) represents another interesting problem to be explored, among others.

 %\appendix
%\subsection{Resultados principais}

\appendix
	
\section{Model}\label{model}

The model deduced in \cite{Batista1} is restated here, with the only difference being that some of the physical parameters are considered dependent on the spatial variable $x$ rather than being constant.

The porous medium is presumed horizontal, one-dimensional, and
comprised of $n$ parallel layers (see Fig.~\ref{meio-poroso}), each one with an initially available concentration of a solid fuel, such as coke.

\begin{figure}[h]
	\begin{center}
		{\includegraphics*[width=3.0in]{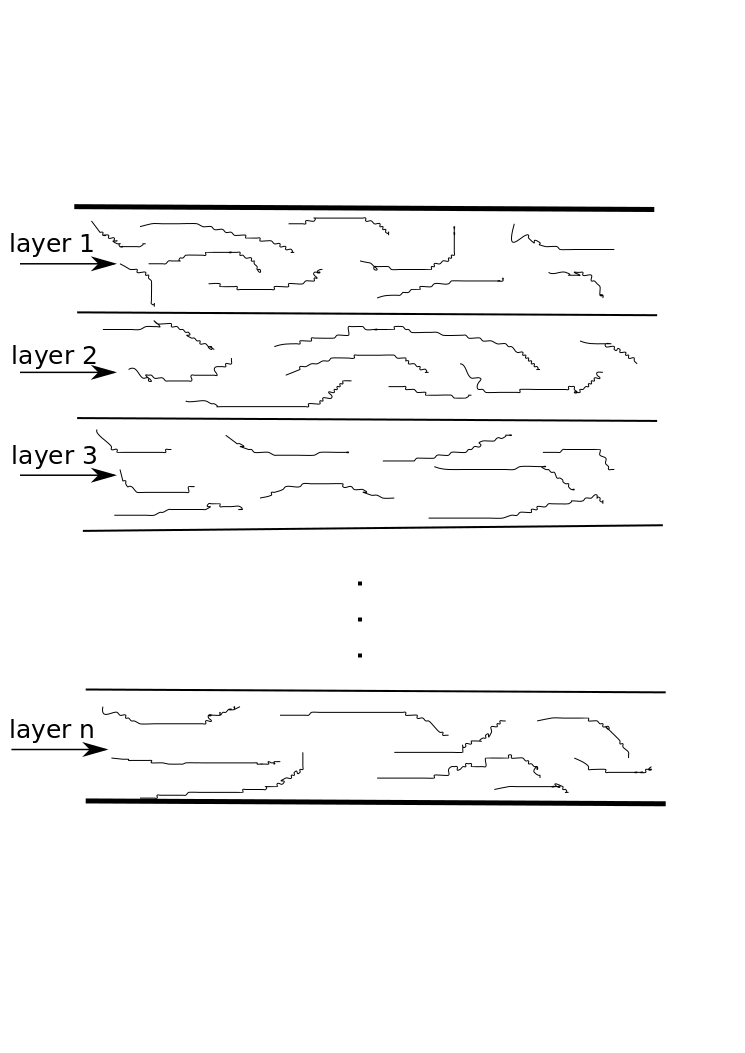}
			\begin{center}
				\vspace{-1.0in} 
				\caption {Porous medium with $n$ parallel layers}
				\label{meio-poroso}
			\end{center}
		}
	\end{center}
\end{figure} 
 
The chemical reaction in each layer takes the simple form
\begin{align}
[\text{solid reactant] $+$ [gaseous reactant] $\rightarrow$
	[gaseous product] $+$ [heat]} \,.
\end{align}

To formulate the balance equations, we presume that gas, rock matrix, and fuel are locally in thermal equilibrium in each layer at all times. Therefore, only one temperature is used for the energy balance in each layer. Porosity and conductivity in each layer are assumed to be functions of the space variable. The effects of radiation, viscous dissipation, and work by 
pressure changes are neglected. However, a reaction rate, longitudinal heat conduction, 
heat transfer between the layers, and heat loss to the surrounding rock formation are taken into consideration.
Subscripts $g$, $r$, and $c$ refer to oxygen, rock, and coke,
respectively, and subscript $s$ refers to an entire layer.
Subscript $i$, $i = 1,\ldots, n$ designate the $i$th layer.

In the $i$th layer, the state variables
depending on $(x, t)$ are temperature $T_{i}$, fuel
concentration $\eta_{i}$, oxygen mass fraction in the gas
phase $Y_i$, Darcy's velocity $v_i$, and pressure $p_{i}$.

The other relevant quantities in the $i$th layer are
gas density $\rho_{g_i}$, given by a layer-independent equation of state 
$\rho_{g_i} = \rho_g(T_{i}, p_{i})$, rock density $\rho_{r_i}$,
porosity $\phi_i$,  thermal conductivity $\lambda_{i}$,
specific heat of gas, rock, and coke at constant pressure $c_{g_i}$, $c_{r_i}$, and $c_{c_i}$, respectively, rate of coke consumption in the chemical reaction $r_i$, heat of reaction $Q_{h_i}$,
flow resistance appearing in Darcy's law $K_{s_i}$,
which is directly proportional to rock permeability and inversely
proportional to gas viscosity, mass-weighted stoichiometric coefficients
for oxygen and inert gas $m_{o_i}$ and $m_{g_i}$, respectively.
The quantity $m_{g_i}$ may assume positive, negative, or zero values
depending on whether the amount of gas produced by the reaction is
more than, less than, or equal to the amount of gas consumed by it.
The coefficients of heat transfer
between the layers $i$ and $i+1$ are denoted by $Q_i$, $i = 1, \ldots, n-1$. The coefficients of heat transfer
between the external environment and layers $1$ and $n$ are, denoted by $\hat Q_1$ and $\hat Q_2$, respectively. 
Lastly, the external environment's temperature is assumed to remain constant and is denoted by $T_e$.

The following equations are assumed to hold in the $i$th layer, $i=1, \ldots, n$.
\smallskip
 
\noindent{\it Balance of energy}

\noindent For $i = 1$,
+
\begin{align}
\frac{\partial}{\partial t}&\big(\phi_1 \rho_{g_1} c_{g_1} T_1 +
(1-\phi_1) \rho_{r_1} c_{r_1} T_1 + \eta_1 c_{c_1} T_1 \big) = 
- \frac{\partial}{\partial x} \left( \rho_{g_1} c_{g_1} v_1 T_1 \right)
\nonumber
\\ 
& \phantom{---} + Q_{h_1} r_1 - \hat Q_1 (T_1 - T_e) + Q_1 (T_2 - T_1) + 
\frac{\partial}{\partial x}\big(\lambda_{s_1} \frac{\partial T_1} {\partial x} \big).
\label{be1}
\end{align}
For $i = 2, \ldots, n-1$,
\begin{align}
\frac{\partial}{\partial t}&\big(\phi_i  \rho_{g_i} c_{g_i} T_i +
(1-\phi_i) \rho_{r_i} c_{r_i} T_i + \eta_i c_{c_i} T_i \big) = 
- \frac{\partial}{\partial x} \left( \rho_{g_i} c_{g_i} v_i T_i \right)
\nonumber
\\ 
&\phantom{---} + Q_{h_i} r_i - Q_{i-1} (T_i - T_{i-1}) + Q_i (T_{i+1} - T_i) + 
\frac{\partial}{\partial x}\big(\lambda_{s_i} \frac{\partial T_i} {\partial x} \big). 
\label{be2}
\end{align}
For $i = n$,
\begin{align}
\frac{\partial}{\partial t}\big(\phi_n& \rho_{g_n} c_{g_n} T_n +
(1-\phi_n) \rho_{r_n} c_{r_n} T_n + \eta_n c_{c_n} T_n \big) = 
- \frac{\partial}{\partial x} \left( \rho_{g_n} c_{g_n}v_n T_n \right)
\nonumber
\\ 
&\phantom{-} + Q_{h_n} r_n - Q_{n-1} (T_n - T_{n-1}) - \hat Q_2 (T_n - T_e) + 
\frac{\partial}{\partial x}\big(\lambda_{s_n} \frac{\partial T_n} {\partial x} \big).
\label{be3}
\end{align}
\noindent {\it Balance of fuel mass}
\begin{align} \label{bmf1}
\frac{\partial \eta_i}{\partial t} = - r_i \,.
\end{align}
\noindent {\it Balance of oxygen mass}
\begin{align} \label{bmo1}
\frac{\partial}{\partial t}(\phi_i \rho_{g_i} Y_i) +\
\frac{\partial}{\partial x}(\rho_{g_i} v_i Y_i) = -m_i r_i \, .
\end{align}
\noindent {\it Balance of total gas mass}
\begin{align} \label{bmg1}
\frac{\partial}{\partial t}(\phi_i \rho_{g_i}) +
\frac{\partial}{\partial x}(\rho_{g_i} v_i) =
m_{g_i} r_i \, .
\end{align}
\noindent {\it Darcy$^{\prime}$s law}
\begin{align} \label{darcy1}
v_i = - K_{s_i} \frac{\partial p_i}{\partial x} \, .
\end{align}
We consider the conductivity in each layer as an average of the conductivities of the solid materials (rock and coke) and the gaseous phase, given by the expression:
\begin{align}\label{conductivity}
\lambda_{s_i} = (1-\phi_i)\big( (1 - l_i) \lambda_{r_i} + l_i \lambda_{c_i}\big) + \phi_i \lambda_{g_i},
\end{align}
where $ \lambda_{r_i}$, $\lambda_{c_i}$ and $\lambda_{g_i}$ represent the conductivities of rock, coke, and gas phase, 
respectively. The constant $l_i$ is such that $0 \leq l_i \leq 1$. 
The rate of coke consumption in the chemical reaction in each layer is assumed to be given by a version of Arrhenius' law:
\begin{align} \label{reaction-rate1}
r_i = A_{c_i} (Y_i p_{i})^{\alpha} \,\eta_{i}\, e^{-\frac{\hat E}{R T_{i}}} ,
\end{align}
where $ A_{c_i}$ is the Arrhenius constant, $\hat E$ is the activation energy, $\alpha$ is the order of the 
gaseous reaction rate, and $R$ is the gas constant. The values $\hat E$ and $\alpha$ are assumed to be the same for all layers. 

We introduce dimensionless variables for space and time:
\begin{align} \label{semdimensao}
\tilde x = \frac{x}{x^*}, \quad \tilde t = \frac{t}{t^*} \,,
\end{align}
where $x^*$ and $t^*$  are reference values for
space and time, respectively.
The dimensionless Darcy's velocity in the $i$th
layer is given by:
\begin{align} \label{darcy2}
\tilde v_i = \frac{t^*v_i}{x^*}\,.
\end{align}
We also introduce dimensionless variables for temperature,
pressure, gas density, and fuel concentration in the $i$th layer,
\begin{align} \label{semdimensao-Tp}
\tilde{T}_i= \frac{T_{i}}{T^*}, \quad
\tilde{p}_i= \frac{p_{i}}{p^*}, \quad
\tilde{\rho}_{g_i} = \frac{\rho_{g_i}}{\rho_g^*}, \quad \text{and}
\quad \tilde{\eta}_i = \frac{\eta_{i}}{\eta^o_i},
\end{align}
where $T^*$, $p^*$, and $\rho_g^*$ are the reference values for temperature, pressure,
and gas density, respectively, and $\eta^o_i$
is the initial fuel concentration in the $i$th layer, which is a function of the space variable $x$.
Thus, $\tilde{\eta}_i$ is the fraction of coke remaining in the  $i$th
layer, $0 \le \tilde{\eta}_i \le 1$. The reference gas density $\rho_g^*$
is obtained from the reference temperature and pressure using the state equation.

Defining the dimensionless function 
\begin{align} \label{eq:funcao-h}
h(T, \eta, Y, p) = (Y p)^{\alpha}\, \eta \, e^{-\frac{E}{T}} \, ,
\,\,\,\text{where} \,\,\, E = \frac{\hat E}{R T^*}\,,
\end{align}
the reaction rate function can be write as $r_i = A_{c_i}\,h(T_i, \eta_i, Y_i, p_i)$.
With these dimensionless variables, after dropping the tildes from
Eqs.~\eqref{semdimensao}-\eqref{semdimensao-Tp} and from $\hat T_e = T_e / T^*$, to simplify the notation, the system of 
Eqs.~\eqref{be1}--\eqref{darcy1} becomes
\begin{align}
&\frac{\partial}{\partial t}\Big(\left(\phi_1 \rho_g^* \rho_1 c_{g_1}
+ (1-\phi_1) \rho_{r_1} c_{r_1} + \eta_{1}^o c_{c_1} \eta_1 \right) T_1 \Big) = 
- \frac{\partial}{\partial x} \left( \rho_g^* \rho_1
c_{g_1} v_1 T_1 \right) \phantom{-----------}
\nonumber
\\
&\phantom{------} + \frac{t^* \eta_{1}^o A_{c_1} Q_{h_1} (p^*)^{\alpha}} {T^*}
h(T_1, \eta_1, Y_1, p_1) - t^* \hat Q_1 (T_1 - \hat T_e)   \
\nonumber
\\
&\phantom{------------} + t^* Q_1 (T_2 - T_1) + \frac{t^*}{(x^*)^2} \frac{\partial}{\partial x}
\big(\lambda_{s_1} \frac{\partial T_1}{\partial x}\big),  
\label{be12}
\end{align}
\begin{align}
&\frac{\partial}{\partial t}\Big(\left(\phi_i \rho_g^* \rho_i c_{g_i}
+ (1-\phi_i) \rho_{r_i} c_{r_i} + \eta_{i}^o c_{c_i} \eta_i  \right)
T_i \Big) = - \frac{\partial}{\partial x} \left( \rho_g^* \rho_i
c_{g_i} v_i T_i \right) \phantom{---------}
\nonumber
\\
&\phantom{------} + \frac{t^* \eta_{i}^o A_{c_i} Q_{h_i} (p^*)^{\alpha}} {T^*}
h(T_i, \eta_i, Y_i, p_i) - t^* Q_{i-1} (T_i - T_{i-1})  
\nonumber
\\
&\phantom{------------} + t^* Q_i (T_{i+1} - T_i) + \frac{t^*}{(x^*)^2} \frac{\partial}{\partial x}
\big(\lambda_{s_i} \frac{\partial T_i}{\partial x}\big),  
\label{bei2}
\end{align}
\begin{align}
&\frac{\partial}{\partial t}\Big(\left(\phi_n \rho_g^* \rho_n c_{g_n}
+ (1-\phi_n) \rho_{r_n} c_{r_n} + \eta_{n}^o c_{c_n} \eta_n \right)
T_n \Big) = - \frac{\partial}{\partial x} \left( \rho_g^* \rho_n
c_{g_n} v_n T_n \right) \phantom{----------} 
\nonumber
\\
&\phantom{------} + \frac{t^* \eta_{n}^o A_{c_n} Q_{h_n} (p^*)^{\alpha}} {T^*}
h(T_n, \eta_n, Y_n, p_n) - t^* Q_{n-1} (T_n - T_{n-1})
\nonumber
\\
&\phantom{-----------} - t^* \hat Q_2 (T_n - \hat T_e) + \frac{t^*}{(x^*)^2} \frac{\partial}{\partial x}
\big(\lambda_{s_n} \frac{\partial T_n}{\partial x}\big),  
\label{ben2}
\end{align}
\begin{align} \label{bmf2}
\frac{\partial \eta_i}{\partial t} = - t^* A_{c_i}(p^*)^{\alpha}\,h(T_i, \eta_i, Y_i, p_i)\,, 
\end{align}
\begin{align} \label{bmo2}
\frac{\partial}{\partial t}(\phi_i \rho_g^* \rho_{g_i} Y_i) +
\frac{\partial}{\partial x}(\rho_g^* \rho_{g_i} v_i Y_i) = -
t^* m_{o_i} \eta^o_i A_{c_i} (p^*)^{\alpha}\,h(T_i, \eta_i, Y_i, p_i)\,,
\end{align}
\begin{align} \label{bmg2}
\frac{\partial}{\partial t}(\phi_i \rho_g^* \rho_{g_i}) +
\frac{\partial}{\partial x}(\rho_g^* \rho_{g_i} v_i) = t^*
m_{g_i} \eta^o_i A_{c_i} (p^*)^{\alpha}\,h(T_i, \eta_i, Y_i, p_i)\,, 
\end{align}
\begin{align} \label{eq:darcy-2}
v_i = - \frac{t^* p^* K_{s_i}}{(x^*)^2 } \frac{\partial p_i}{\partial x} \,. \
\end{align}

Now, let us divide Eqs.~\eqref{be12}, \eqref{bei2}, and \eqref{ben2} by the mean value 
\begin{align} \label{valore-medio}
\bar{\rho_r c_r} = \frac 1n \sum_{i=1}^n \rho_{r_i} c_{r_i}\,,
\end{align}
\eqref{bmf2} by $\eta_{r_i}^o$, and 
\eqref{bmo2} and \eqref{bmg2} by $\rho_g^*$. 
The reason for dividing Eqs. \eqref{be12}, \eqref{bei2}, and \eqref{ben2} by the given constant is that all
coefficients in the resulting equations vary in a range
convenient for numerical computation. For $i = 1, \ldots, n$, we obtain the following dimensionless system:
\begin{align} 
\frac{\partial}{\partial t}&\big((a_1 + b_1 \eta_1) T_1 \big)+
\frac{\partial}{\partial x} \left( \hat c_1 v_1 T_1 \right) = 
\nonumber
\\
&\hat d_1\,h(T_1, \eta_1, Y_1, p_1) - \hat q_1 (T_1 - \hat T_e) + 
q_1 (T_2 - T_1) + \
\frac{\partial}{\partial x} \big(\lambda_1 \frac{\partial T_1}{\partial x}\big)\,,
\label{be13}
\end{align}
\begin{align}
\frac{\partial}{\partial t}&\big((a_i + b_i \eta_i) T_i \big) + 
\frac{\partial}{\partial x} \left(\hat c_i v_i T_i \right) = 
\nonumber
\\
&\hat d_i\,h(T_i, \eta_i, Y_i, p_i) - q_{i-1} (T_i - T_{i-1}) + q_i (T_{i+1} - T_i) + 
\frac{\partial}{\partial x} \big( \lambda_i \frac{\partial T_i}{\partial x}\big)\,,
\label{bei3} 
\end{align}
\begin{align}
&\frac{\partial}{\partial t}\big((a_n + b_n \eta_n ) T_n \big) + 
\frac{\partial}{\partial x} \left(\hat c_n v_n T_n \right) = \phantom{-}\
\nonumber
\\
& \hat d_n\, h(T_n, \eta_n, Y_n, p_n) - q_{n-1} (T_n - T_{n-1}) - \hat q_2 (T_n - \hat T_e) + 
\frac{\partial}{\partial x} \big(\lambda_n \frac{\partial T_n}{\partial x}\big)\,,
\label{ben3} 
\end{align}
\begin{align} \label{bmf3}
\frac{\partial \eta_i}{\partial t} = - \hat A_i \, h(T_i, \eta_i, Y_i, p_i)\,, 
\end{align}
\begin{align} \label{bmo3}
\frac{\partial}{\partial t}(\phi_i \rho_{g_i} Y_i) +
\frac{\partial}{\partial x}(\rho_{g_i} v_i Y_i) = - B_i \, h(T_i, \eta_i, Y_i, p_i)\,,
%B_i = \frac {t^* m_{o_i} \eta^o_i A_{c_i} (p^*)^{\alpha}}{\rho_g^*} h(T_i, \eta_i, Y_i, p_i)\,,
\end{align}
\begin{align} \label{bmg3}
\frac{\partial}{\partial t}(\phi_i \rho_{g_i}) +
\frac{\partial}{\partial x}(\rho_{g_i} v_i) = D_i \, h(T_i, \eta_i, Y_i, p_i)\,,
\end{align}\
\begin{align} \label{darcy3}
v_i = - K_i \frac{\partial
	p_i}{\partial x} \,. 
\end{align}
\noindent where $0 < x < l$, $t > 0$. In Eq.~\eqref{bei3} 
$i = 2, \ldots, n-1$; in Eqs.~\eqref{bmf3}, \eqref{bmo3}, \eqref{bmg3}, and \eqref{darcy3} $i = 1, \ldots, n$. The coefficients are given by
\begin{align} \label{parametros1}
a_i=\frac{\phi_i \rho_g^* \rho_{g_i} c_{g_i} + (1-\phi_i) \rho_{r_i}
	c_{r_i}} {\bar{\rho_r c_r}}, \,\, b_i = \frac{\eta^o_i
	c_{c_i}}{\bar {\rho_r c_r}}, \,\, \hat c_i = \frac{\rho_g^* \rho_{g_i} c_{g_i}}{\bar {\rho_r c_r}}\,,
\end{align}
\begin{align} \label{parametros2}
\hat d_i = \frac{A_i \eta^o_i Q_{h_i}} {T^* \bar {\rho_r c_r}}, \,\,
\hat A_i = t^* A_{c_i} (p^*)^{\alpha}, \,\,\,
\lambda_i = \frac{t^* \lambda_{s_i}}{(x^*)^2 \bar {\rho_r c_r}}, \,\,\, 
B_i = \frac{m_{o_i} A_i \eta^o_i}{\rho_g^*}\,, 
\end{align}
\begin{align} \label{parametros3} 
D_i = \frac{m_{g_i} A_i \eta^o_i}{\rho_g^*}, \,\,\, K_i = \frac{t^* p^* K_{s_i}} {(x^*)^2}, \,\, q_i = \frac{t^* Q_i}{\bar {\rho_r c_r}}, \quad \hat q_1 = \frac{t^* \hat Q_1}{\bar {\rho_r c_r}}, 
\quad \hat q_2 = \frac{t^* \hat Q_2}{\bar {\rho_r c_r}} \,,
\end{align} 
where in Eqs.~\eqref{parametros1}, \eqref{parametros2}, and \eqref{parametros3} $i = 1, \ldots, n$, except for $q_i$ in 
Eq.~\eqref{parametros3}, where $i = 1, \ldots, n-1$.

The coefficients defined in Eqs.~\eqref{parametros1}--\eqref{parametros3}
are dependent on physical quantities in the layers. In this work, we consider all of them to be dependent on the spatial variable $x$ rather than being constants.

Therefore, $a_i$, $b_i$, $c_i$, $d_i$, $\lambda_i$, 
$A_i$, $B_i$, $D_i$, $q_i$, $\hat q_1$, $\hat q_2$, and 
$K_i$, are all depending on $x$, and they are all non-negative, except $D_i$, which, depending on the 
stoichiometric coefficient $m_{g_i}$, may be positive, negative, or zero.

A convenient way to analyze the system of Eqs.~\eqref{be13}-\eqref{darcy3} is to assume, initially, that 
the fluids are incompressible. This allows us to disregard volume and pressure changes due to chemical reaction.
These assumptions simplify our equations and isolate the primary temperature effects.
Therefore, for  $i = 1,\ldots, n$, we assume that $\rho_i$ is a constant average value, denoted by $\bar \rho_i$, 
independent on temperature and pressure, and $m_{g_i}=0$, which implies that $\tilde D_i=0$.

We are interested in the physical situation in which oxygen is injected into the porous medium at $x=0$, 
and all the solid fuel burns, causing a reaction front to propagate to the right. As $\tilde D_i=0$, from Eq.~\eqref{bmg3}, we have $\frac{\partial v_i}{\partial x} = 0$.
Thus, $v_i$ depends on time only. One can relate it to boundary conditions at the injection end. For simplicity, we assume it to be constant. Also, from Eq.~\eqref{darcy3}, we see 
that $p_i$ can be easily calculated using the injection pressure.

With these simplifications, Eqs.~\eqref{bmg3} and \eqref{darcy3} are automatically satisfied, and 
Eqs.~\eqref{be13}-\eqref{bmf3} are coupled with Eq.~\eqref{bmo3} only by the factor $(Y_i p_i)^{\alpha}$ in the function $h$. For simplicity, we take this factor as a known constant 
average value. Then, from Eqs.~(\ref{be13}) and (\ref{bmf3}), we obtain the following system modeling the temperatures and the fuel concentrations only. To simplify the notation, we rename the temperature and the fuel concentration in the $n$ layers, $u_i  = T_i$ and $y_i = \eta_i$, for $i = 1, \ldots, n$, we also rename $u_e = \hat T_e$.
\begin{align} 
\frac{\partial}{\partial t}\big((a_1 + b_1 y_1) u_1 \big)+
\frac{\partial}{\partial x} \left( c_1 u_1 \right) = \phantom{---}
\nonumber
\\
d_1\,y_1 g(u_1) - \hat q_1 (u_1 - u_e) &+ 
q_1 (u_2 - u_1) + 
\lambda_1 \frac{\partial^2 u_1}{\partial x^2} ,
\label{be14}
\end{align}
\begin{align}
\frac{\partial}{\partial t}\big((a_i + b_i y_i) u_i \big) + 
\frac{\partial}{\partial x} \left(c_i u_i \right) = \phantom{---}
\nonumber
\\
d_i\,y_i g(u_i) - q_{i-1} (u_i - u_{i-1}&) + q_i (u_{i+1} - u_i) + 
\lambda_i \frac{\partial^2 u_i}{\partial x^2}\,,
\label{bei4} 
\end{align}
\begin{align}
\frac{\partial}{\partial t}\big((a_n + b_n y_n ) u_n \big) + 
\frac{\partial}{\partial x} \left(c_n u_n \right) = 
\nonumber
\\
d_n\, y_n g(u_n) - q_{n-1} (u_n - &u_{n-1}) - \hat q_2 (u_n - u_e) + \lambda_n \frac{\partial^2 u_n}{\partial x^2}\,,
\label{ben4} 
\end{align}
\begin{align} \label{bmf4}
\frac{\partial y_i}{\partial t} = - A_i y_i g(u_i)\,, 
\end{align}
where $0 < x < l$, $t > 0$. The function $g$ 
is related to the Arrhenius law, and it is given by
\begin{equation}\label{darcy} 
g(\theta)=\left\{\begin{array}{c}
e^{-\frac{E}{\theta}},\;\;\mbox{se}\;\; \theta>0\\
0,\;\;\mbox{se}\;\; \theta\leq 0\,,
\end{array}\right.
\end{equation}
where $E>0$.

The new coefficients are
\begin{align}\label{new-coeficients}
c_i = \hat c_i \,v_i, \quad d_i = (Y_i p_i)^{\alpha} \hat d_i , \quad  A_i = (Y_i p_i)^{\alpha} \hat A_i, 
\quad i = 1, \ldots, n. 
\end{align}
The other coefficients $a_i$, $b_i$, and $\lambda_i$, for $i = 1, \ldots, n$, $q_i$, for  $i = 1, \ldots, n-1$, 
and $\hat q_i$, for $i = 1,\,\, 2$, are the same as defined in Eqs.~\eqref{parametros1}-\eqref{parametros3}. 
In this study, we consider all these coefficients as depending on the spatial variable $x$.

This work regards the fuel concentration $y_i = y_i(x, t)$ as a known function. Thus, from \eqref{be14}-\eqref{ben4} we can write the following system modeling only the temperatures in the $n$ layers,
\begin{align} \label{vetorial-model}
u_t + L(t)u = f(x, t, u), \,\,\, x \in \R, \,\,\, t >0,	
\end{align}
where $u = (u_1, \ldots, u_n)$ is the vector of the unknown temperatures, 
$L(t)u = \big(L_1(t)u_1, $ $\ldots, L_n(t)u_n \big)$, is the partial differential operator, each component defined by
\begin{equation}\label{operatorLi}
L_i(t)u_i := -\alpha_i(x, t)\,\partial_{xx}u_i + \beta_i(x, t)\,\partial_{x}u_i,\,\,\,
i = 1, \ldots, n,
\end{equation}
being 
\begin{equation}\label{coeficients}
\alpha_i(x, t) = \frac{\lambda_i}{a_i + b_i y_i(x, t)}, \,\,\, \mbox{and} \,\,\, \beta_i(x, t) = 
\frac{c_i}{a_i + b_i y_i(x, t)}.
\end{equation}
The source function is $f = (f_1, \ldots, f_n)$, where 
\begin{align} \label{reaction-function} 
	&f_1(x, t, u) = \frac{-(c_1)_x\,u_1 }{a_1 + b_1 y_1} + 
	\frac{(K_1 b_1 u_1+d_1)y_1\,g(u_1)}{a_1 + b_1 y_1} + \frac{q_1(u_2 - u_1)}{a_1 + b_1 y_1}
	\nonumber
	\\   
	& \hspace{0.8in} - \frac{\hat q_1(u_1 - u_e)}{a_1+b_1 y_1}\,,
	\nonumber \\
	&f_i(x, t, u) = \frac{-(c_i)_x\,u_i }{a_i + b_i y_i} + 
	\frac{(K_i b_i u_i+d_i)y_i\,g(u_i)}{a_i + b_i y_i} - \frac{q_{i-1}(u_i - u_{i-1})}{a_i + b_i y_i}
	\\
	& \hspace{0.8in} + \frac{q_i(u_{i+1} - u_i)}{a_i + b_i y_i}, \quad i = 2, \ldots, n-1\,,
	\nonumber \\
	&f_n(x, t, u) = \frac{-(c_n)_x\,u_n }{a_n + b_n y_n} + 
	\frac{(K_n b_n u_n+d_n)y_n\,g(u_n)}{a_n + b_n y_n} - \frac{q_{n-1}(u_n - u_{n-1})}{a_n + b_n y_n}
	\nonumber
	\\
	& \hspace{0.8in} - \frac{\hat q_2(u_n - u_e)}{a_n + b_n y_n}\,. 
	\nonumber 
\end{align}

\bibliographystyle{plain}
\bibliography{combustion}

\begin{thebibliography}{10}

\bibitem{Ado}
M.~R. Ado.
\newblock Improving oil recovery rates in thai in-situ combustion process using
  pure oxygen.
\newblock {\em Upstream Oil and Gas Technology}, 6:100032, 2021.

\bibitem{Alarcon}
E.~A. Alarcon, M.~R. Batista, A.~Cunha, J.~C.~Da Mota, and R.~A. Santos.
\newblock Application of the semigroup theory to a combustion problem in a
  multi-layer porous medium.
\newblock {\em arXiv. https://doi.org/10.48550/arXiv.2206.06766. Accepted for
  publication in ``Journal of Applied Analysis and Computation''}, 2023.

\bibitem{Banerjee}
A.~Banerjee and D.~Paul.
\newblock Developments and applications of porous medium combustion: A recent
  review.
\newblock {\em Energy}, 221:119868, 2021.

\bibitem{Batista1}
M.~R. Batista and J.~C. Da~Mota.
\newblock Monotone iterative method of upper and lower solutions applied to a
  multilayer combustion model in porous media.
\newblock {\em Nonlinear Analysis: Real World Applications}, 58:103223, 2021.

\bibitem{Batista2}
M.~R. Batista, J.~C. Da~Mota, and R.~A. Santos.
\newblock Classical solution for a nonlinear hybrid system modeling combustion
  in a multilayer porous medium.
\newblock {\em Nonlinear Analysis: Real World Applications}, 63:103406, 2022.

\bibitem{DaMota-Marcelo}
Da~Mota~J. C. and Santos~M. M.
\newblock An application of the monotone iterative method to a combustion
  problem in porous media.
\newblock {\em Nonlinear Analysis: Real World Applications}, 12:1192--1201,
  2010.

\bibitem{Chapiro}
G.~Chapiro and D.~Marchesin.
\newblock The effect of thermal losses on traveling waves for in-situ
  combustion in porous medium.
\newblock In {\em Journal of Physics: Conference Series}, volume 633, page
  012098. IOP Publishing, 2015.

\bibitem{Coats}
K.~H. Coats.
\newblock In-situ combustion model.
\newblock {\em Society of Petroleum Engineers Journal}, 20:533--555, 1980.

\bibitem{Cunha}
A.~Cunha and E.~Alarcon.
\newblock The {IVP} for the evolution equation of wave fronts in chemical
  reactions in low-regularity sobolev spaces.
\newblock {\em Journal of Evolution Equations}, 21(1):921--940, 2021.

\bibitem{DaMota-Marcelo-Ronaldo}
J.~C. Da~Mota, M.~M. Santos, and Santos~R. A.
\newblock The cauchy problem for a combustion model in a porous medium with two
  layers.
\newblock {\em Monatshefte fur Mathematik}, 188:131--162, 2019.

\bibitem{DaMota-Schecter}
J.~C. Da~Mota and S.~Schecter.
\newblock Combustion fronts in a porous medium with two layers.
\newblock {\em Journal of dynamics and differential equations}, 18(3):615--665,
  2006.

\bibitem{Fatehi}
H.~Fatehi and Xue-Song Bai.
\newblock A comprehensive mathematical model for biomass combustion.
\newblock {\em Combustion Science and Technology}, 186, 04 2014.

\bibitem{Gao}
C.~Gao and H.~Sun.
\newblock {\em Well Test Analysis for Multilayered Reservoirs with Formation
  Crossflow}.
\newblock Gulf Professional Publishing, Cambridge, MA, USA, 2017.

\bibitem{Gharehghani}
A.~Gharehghani, K.~Ghasemi, M.~Siavashi, and S.~Mehranfar.
\newblock Applications of porous materials in combustion systems: A
  comprehensive and state-of-the-art review.
\newblock {\em Fuel}, 304:121411, 2021.

\bibitem{Iorio}
R.~J. I\'orio and V.~M. I\'orio.
\newblock {\em Fourier Analysis and Partial Differential Equations}.

\bibitem{Kato3}
T.~Kato.
\newblock Linear evolution equations of “hyperbolic” type, ii.
\newblock {\em Journal of the Mathematical Society of Japan}, 25(4):648--666,
  1973.

\bibitem{Kato4}
T.~Kato.
\newblock Quasi-linear equations of evolution, with applications to partial
  differential equations.
\newblock In {\em Spectral theory and differential equations}, pages 25--70.
  Springer, 1975.

\bibitem{Kato5}
T.~Kato.
\newblock {\em Abstract differential equations and nonlinear mixed problems}.
\newblock Accademia Nazionale Dei Lincei Scuola Normale Superiore. Lezioni
  Fermiane, Pisa, Italy, 1985.

\bibitem{Lefkovits}
M.~J. Lefkovits, P.~Hazebroek, E.~E. Allen, and C.~S. Matthews.
\newblock A study of the behavior of bounded reservoirs composed of stratified
  layers.
\newblock {\em Society of Petroleum Engineers Journal}, 1(01):43--58, March
  1961.

\bibitem{Mujeebu}
M.~A. Mujeebu, M.~Z. Abdullah, M.~Z.~Abu Bakar, A.~A. Mohamad, R.~M.~N. Muhad,
  and M.~K. Abdullah.
\newblock Combustion in porous media and its applications – a comprehensive
  survey.
\newblock {\em Journal of Environmental Management}, 90(8):2287--2312, 2009.

\bibitem{pazy}
A.~Pazy.
\newblock Semigroups of linear operators and applications to partial
  differential equations.
\newblock {\em Appl. Math. Sci}, 44, 1983.

\bibitem{Sarathi}
P.~S. Sarathi.
\newblock {\em In-situ combustion handbook - principles and practices}.
\newblock DOE/PC/91008-0374 OSTI ID: 3175, University Libraries, UNT Digital
  Library, 1999.

\bibitem{Trimis}
D.~Trimis and F.~Durst.
\newblock Combustion in a porous medium-advances and applications.
\newblock {\em Combustion Science and Technology}, 121(1-6):153--168, 1996.

\end{thebibliography}
\end{document}